\documentclass[11pt,a4paper]{amsart}

\usepackage[foot]{amsaddr}
\usepackage{graphicx,datetime}

\usepackage{amsthm} 
\usepackage{amssymb} 
\usepackage{amsmath} 
\usepackage[hidelinks]{hyperref} 
\usepackage[capitalize,nameinlink]{cleveref} 
\usepackage{tikz}
\usetikzlibrary{decorations.markings}
\usepackage{pgfplots}
\usepackage{tikz-cd} 
\usepackage{comment}
\usepackage{subcaption}

\newtheorem{theorem}{Theorem}[section]
\newtheorem*{theorem*}{Theorem}
\newtheorem{proposition}[theorem]{Proposition}

\newtheorem*{lemma*}{Lemma}

\newtheorem{corollary}[theorem]{Corollary}

\newtheorem*{conjecture*}{Conjecture}

\theoremstyle{definition}
\newtheorem{definition}[theorem]{Definition}
\newtheorem*{definition*}{Definition}
\newtheorem{example}[theorem]{Example}

\newtheorem{remark}[theorem]{Remark}
\newtheoremstyle{noparens}{}{}{\itshape}{}{\bfseries}{.}{ }{\thmname{#1}\thmnumber{ #2}\thmnote{ {\mdseries #3}}}
\theoremstyle{noparens}

\newtheorem*{theoremnop*}{Theorem}

\newtheoremstyle{noparensdef}{}{}{}{}{\bfseries}{.}{ }{\thmname{#1}\thmnumber{ #2}\thmnote{ {\mdseries #3}}}
\theoremstyle{noparensdef}

\newcommand{\R}{\mathbb{R}} 
\newcommand{\Z}{\mathbb{Z}} 
\newcommand{\N}{\mathbb{N}} 
\renewcommand{\S}{\mathbb{S}} 

\newcommand{\id}{e}

\newcommand{\<}{\left\langle}
\renewcommand{\>}{\right\rangle}

\makeatletter
\newcommand*\owedge{\mathpalette\@owedge\relax}
\newcommand*\@owedge[1]{%
  \mathbin{%
    \ooalign{%
      $#1\m@th\bigcirc$\cr
      \hidewidth$#1\m@th\wedge$\hidewidth\cr
    }%
  }%
}
\makeatother

\tikzset{->-/.style={decoration={
  markings,
  mark=at position #1 with {\arrow{>}}},postaction={decorate}}}
\tikzset{->>-/.style={decoration={
  markings,
  mark=at position #1 with {\arrow{>>}}},postaction={decorate}}}
\tikzset{->>>-/.style={decoration={
  markings,
  mark=at position #1 with {\arrow{>>>}}},postaction={decorate}}}
\tikzset{->>>>-/.style={decoration={
  markings,
  mark=at position #1 with {\arrow{>>>>}}},postaction={decorate}}}

\usepackage{euscript,color}

\begin{document}
\title{Fundamental regions for non-isometric group actions}
\thanks{
The first author was supported by 
 the Australian Research
Council (Discovery Program DP190102360). The second author 
was supported through an Australian Government Research Training Program  Scholarship (RTP Fees Offset).
 }
\author{Thomas Leistner}
\address[Thomas Leistner, corresponding author]{School of Computer and Mathematical Sciences, University of Adelaide, SA 5005, Australia}
\email{thomas.leistner@adelaide.edu.au}
\author{Stuart Teisseire}\address[Stuart Teisseire]{Department of Mathematics, University of Auckland, Private Bag 92019, Auckland~1142, New Zealand}
\email{stuart.teisseire@gmail.com}


\begin{abstract}

We generalise results about isometric group actions on metric spaces and their fundamental regions to the context of merely continuous group actions. In particular, we obtain results that yield the relative compactness of a fundamental region for a cocompact group action.
As a consequence, 
we obtain a criterion for a cocompact cyclic group of semi-Riemannian homotheties to be inessential.

 \end{abstract}

\maketitle

\maketitle

\section{Introduction}
Fundamental regions and fundamental domains are  important concepts when studying group actions on vector spaces, manifolds or topological spaces. They incorporate essential information about the action itself and about the space of orbits.
	The notion of a fundamental polygon for the elliptic modular group goes back to  Gauss and Dedekind and was later extended by Poincar\'{e} to Fuchsian groups, see \cite[Section 6.9]{Ratcliffe2019HyperbolicManifolds} for historical notes. Consequently the notion 
 has been extensively studied and used in  hyperbolic geometry and more generally in the context of isometric group actions on metric spaces (and our main reference for this will be \cite{Ratcliffe2019HyperbolicManifolds}). Beyond the context of isometric group actions on metric spaces, fundamental domains have been studied for projective actions of Coxeter groups \cite{Vinberg71,Davis08}, for 
 affine actions of $\R^n$ (in relation to Auslander's conjecture), and in particular for groups acting by isometries on Minkowski space, see surveys  \cite{Abels01,DancigerDrummGoldman20}. Extending results for Minkowski space, isometric actions on certain conformally flat Lorentzian manifolds were analysed in \cite{DancigerGueritaudKassel16} for anti-de Sitter space and \cite{CharetteFrancoeurLareau-Dussault14} for the Einstein Universe.
 
Our motivation to study more general group actions on the one hand originates from questions in \cite{LeistnerTeisseire22}, where we study conformal group actions on non-conformally flat symmetric Lorentzian manifolds and their compact quotients, and on the other hand from 
an interest in the general context of continuous group actions on  topological spaces. 
For group actions that are proper and smooth on a smooth manifold, an 
  auxiliary invariant Riemannian metric exists   that turns the action into an isometric action for the metric induced by the Riemannian metric \cite[Theorem 2]{Koszul65}. In fact, the existence of an invariant metric is equivalent to properness \cite{AbelsManoussosNoskov11}. 
 In general however, such an auxiliary Riemannian metric is not unique, and bears no relation to the geometric structure of interest, such as a semi-Riemannian metric or a conformal structure, and in our applications there was no direct way to relate both. Hence, we avoid making use of this approach.
 
For our motivation from conformal geometry recall that a diffeomorphism $\phi$ of a semi-Riemannian manifold $(M,g)$ is a {\em conformal transformation} if there is a smooth function $f$ such that $\phi^*g=\mathrm{e}^{2f} g$,   a {\em homothety} if $f$ is constant, and  {\em non-trivial} if $\phi $ is not an isometry.
A conformal transformation is called {\em inessential} if there is a metric $\hat g$ that is conformally equivalent to $g$ and for which $\phi$ is an isometry.  Otherwise $\phi$ is {\em essential}. We are interested in the question when a homothety is (in)essential. 
 It can easily be shown that a non-trivial homothety with a fixed point is essential, see \cite[Proposition 2.5]{LeistnerTeisseire22}.
For example, a constant scaling $\phi(x)= \lambda x$ of Euclidean space $\R^n$, with $\lambda\not=\pm1$,  is an essential homothety for the Euclidean metric $g$. However, if the origin is removed from $\R^n$, $\phi$ acts freely and is no longer essential, as it is an isometry for the metric $\|x\|^{-2} g$. When analysing this phenomenon in a more general context, using the fundamental domain of the group generated by a homothety $\phi$,  the following property turns out to be relevant: a fundamental region $R$ for a group $G$ acting continuously on a topological space $X$ is \emph{strongly locally finite} if its closure $\overline R$ has an open  neighbourhood $U$, such that $\{g U\}_{g\in G}$ is a locally finite family of sets. About the relation between this notion and essentiality  we observe the following, with a proof provided in Section~\ref{appl}.
	\begin{proposition}\label{prop_homoth-essential-implies-no-FSA-FR}
		Let $(M,g)$ be a semi-Riemannian manifold and  $\phi$  a non-trivial homothety of $(M,g)$.
		If the group  generated by $\phi$ admits a strongly locally finite fundamental region, then $\phi$ is not essential.
	\end{proposition}
The property of being strongly locally finite is indeed strictly stronger than the well-known property of local finiteness of  fundamental domain, as we will show by example in Section~\ref{definitions}, where also the relevant definitions from \cite{Ratcliffe2019HyperbolicManifolds} are recalled.
For actions on a locally compact metric space that are discontinuous and cocompact, local finiteness of a fundamental domain  is {\em equivalent} to its relative compactness, see \cite[Theorem 6.6.9]{Ratcliffe2019HyperbolicManifolds}.
Here we will generalise this result to continuous actions as follows.
 \begin{theorem}\label{main_theorem}
 Let $G$ be a group of homeomorphisms of a locally compact topological space $X$ so that $X/G$ is compact, and let $R$ be a fundamental region.
 Then the following are equivalent:
	\begin{enumerate}
		\item\label{locfinThm} $R$ is locally finite,
		\item\label{strlocfinThm} $R$ is strongly locally finite,
		\item\label{fsaThm} $\overline R$ has a neighbourhood $U$ that is met by  $gU$  for only finitely many $g\in G$.
	\end{enumerate}
 Moreover, if $R$ is	
locally finite, then $\overline R$ is compact. If in addition,  
 $G$ acts  properly discontinuously and $X$ is  first countable and Hausdorff, 
the the converse holds, i.e., if $\overline R$  is compact, then $R$ is locally finite.
 \end{theorem}
 En route to this theorem, in Section~\ref{result-sec} we prove several results about fundamental regions of properly discontinuous or cocompact group actions that, whenever possible, generalise results for isometric actions  as in \cite[Section 6]{Ratcliffe2019HyperbolicManifolds}. We believe that these results  are inherently interesting in their own right. We also provide several examples that  illustrate the distinction of the equivalent properties in Theorem~\ref{main_theorem} when its assumptions are not satisfied.
 
 Returning to the original problem of analysing when a homothety is (in)essential, we can use Proposition~\ref{prop_homoth-essential-implies-no-FSA-FR} and Theorem~\ref{main_theorem} to obtain the following result.
 
 \begin{corollary}\label{introcor}
	Let $(M,g)$ be a semi-Riemannian manifold with  a non-trivial homothety $\phi$ and let $R$ be a fundamental region for $G=\langle \phi\rangle $. If $G$ acts cocompactly and  $R$ is locally finite, then   $\phi$ is inessential, and hence has no fixed points.
	In particular, if  $G$ acts properly discontinuously and $\overline{R}$ is compact, then  $\phi$ is inessential.
	\end{corollary}
The motivation for this result comes the aforementioned example of Euclidean space $\R^n$ with  homothety $\phi:x\mapsto \lambda x$, which is essential, but has a fixed point and hence  does not act properly discontinuously. If the origin is removed, $\langle \phi\rangle$ acts properly discontinuously and with an annulus as  fundamental domain, so that $\R^n/\langle \phi\rangle$  is a torus, but $\phi$ is no longer essential. The corollary shows that this phenomenon persists for arbitrary semi-Riemannian manifolds.
\subsection*{Acknowledgements} 
	The results in this paper are contained in the second author's MPhil thesis \cite{stuart-thesis}, which was written under supervision of the first author. We thank Michael Eastwood for co-supervising and for helpful discussions. We also like to thank Nick Buchdahl for valuable comments.
Finally, we would like to thank Fanny Kassel and Pierre-Louis Blayac for very helpful comments that not only improved the presentation but also the results of this paper substantially.
As a result, this final version is substantially different from the first version \cite{LeistnerTeisseire21}, in which the notion of a finitely self adjacent fundamental region played a central role, see Remark~\ref{fsa-remark} for further explanation.

\section{Definitions, examples and existence}\label{definitions}
Let $G$ be a topological group acting continuously on a topological space $X$. We denote by $X/G$ the orbit space of this action, endowed with the final topology.

We will say that a set $U\subset X$ \emph{meets} a set $V\subset X$ if their intersection $U\cap V$ is non-empty.
The action of $G$ on $X$ is said to be \emph{properly discontinuous} if any pair of points $x,y\in X$ have open neighbourhoods $U$ and $V$ respectively, such that $gU$ meets $V$ for only finitely many $g$. This implies in particular that the stabiliser of a point is finite.

There are many inequivalent notions of properly discontinuous, which are equivalent for Hausdorff, first countable spaces. In Kapovich's note \cite{kapovich2023note}, Theorem 11 gives a nice selection of these various notions.
Ratcliffe \cite[Section~5.3]{Ratcliffe2019HyperbolicManifolds} defines a discontinuous action by requiring  that
compact sets meet their $G$-translates for only finitely many $g$.
We have instead chosen to use the slightly stronger notion defined in the previous paragraph; firstly because it seems the most natural in our proofs, and secondly because in the context of studying locally finite fundamental regions, which we will define now, it turns out to be harmless, see \cref{thmproperly-discontinuous-is-harmless}.

An open set $R\subset X$ is a \emph{fundamental region} for the action if
\begin{itemize}
	\item $G$-translates of $R$ are mutually disjoint, i.e.~if $gR$ meets $R$, then $g$ is the identity,
	\item and $G$-translates of its closure $\overline R$ cover $X$.
\end{itemize}
We note that some authors impose the additional assumption that $R$ is the interior of its closure, e.g. \cite[Section 6]{kapovich2023note}. We will not use this additional assumption in any of the results of this paper, so we have not included it here.

Commonly, authors work with connected fundamental regions. A connected fundamental region is called a \emph{fundamental domain}. We will not use the connectedness assumption in any of the results of this paper, so we use fundamental regions instead.

A fundamental region $R$ is \emph{locally finite} if $\{g\overline R\}_{g\in G}$ is a locally finite family of sets, i.e.~every point $x\in X$ has an open neighbourhood $U$, such that $U$ meets $g\overline R$ for only finitely many $g\in G$. By the properties of the closure, this is equivalent to $\{gR\}_{g\in G}$ being a locally finite family of sets.

We will now define the notion of `strongly locally finite', which was mentioned in the introduction.
\begin{definition}\label{slf}
A fundamental region $R$ is \emph{strongly locally finite} if its closure $\overline R$ has an open neighbourhood $U$ such that $\{g U\}_{g\in G}$ is a locally finite family of sets. In this case we call $U$ a {\em strongly locally finite neighbourhood} of $\overline R$.
\end{definition}
Note that, if $U$ is an open neighbourhood of $\overline{R}$, then $gU$ is a neighbourhood of $g\overline R$, so if a fundamental region is strongly locally finite  it is also locally finite.
It is not clear to the authors if  strong local finiteness is equivalent to local finiteness when $X$ is Hausdorff, but the properties are distinct for merely $T_0$ spaces, as the following example shows.
\begin{example}
	Consider a topology on $X:= \Z^2$ generated by the singletons $\{(i,i)\}$ and  the    sets
	$$ \qquad \{(i,j),(i,i),(j,j)\}, \quad \text{  for any $i\neq j\in \Z$.}$$
The topology of  $X$ induces the discrete topology on the subset  of the diagonal elements, but with the points $(i,j)$ acting as a boundary between $(i,i)$ and $(j,j)$. The space $X$ is $T_0$, but not $T_1$.

	Consider the action of the integers on $X$ by
	$$k\cdot (i,j) := (i+k, j+k),$$
and the fundamental region $R:= \{(0,0)\}$. Since its closure is
	$$\overline R = \{(0,i)\ |\ i\in \Z\} \cup \{(i,0)\ |\ i\in \Z\},$$
	this is indeed a fundamental region. With 
	$$k\cdot \overline R = \{(k,i)\ |\ i\in \Z\} \cup \{(i,k)\ |\ i\in \Z\},$$ it follows that $R$ 
 is locally finite.  
However, since every open neighbourhood of $\overline R$ contains the entire diagonal, $R$ is not strongly locally finite.
\end{example}

\begin{remark}[Finitely self adjacent fundamental regions]  \label{fsa-remark}
In \cite{stuart-thesis} and  \cite{LeistnerTeisseire21}  a fundamental region $R$ that satisfies property~(\ref{fsaThm})  in Theorem~\ref{main_theorem} was called a {\em finitely self adjacent} fundamental region. We were informed by Pierre-Louis Blayac that stronger versions of some results in \cite{LeistnerTeisseire21} should be true that do not require this notion, which lead us to the statement of Theorem~\ref{main_theorem}. 
Also in   \cite{stuart-thesis} and  \cite[Section 2]{LeistnerTeisseire21} the example of a finitely self adjacent fundamental region  that is not locally finite was given in detail. This example is obtained by the universal cover $X$ of the punctured Euclidean plane $\R^2/\Z^2$ with a group $G$ acting on $X$ that is generated by infinitely many reflections. Pierre-Louis Blayac also pointed out to us several related and simpler examples of locally finite fundamental regions that are not finitely self adjacent, such as the following.

\begin{example}
 Let $X=\R\times \N$ and $G=\Z$ acting via $k\cdot (x,n)= (x+k,n)$. This action is free, properly discontinuous and by isometries. Then it is easy to see that 
\[R=\bigcup_{n\in \N} \left( (0,\tfrac{1}{2})\cup (\tfrac{1}{2}+n, 1+n)\right)\times \{n\}\]
is a fundamental region, that is strongly locally finite, and hence locally finite. However $R$ is not  finitely self adjacent. Indeed, for each $k\in \N$ the  $-k$ translate of $\overline{R}$, $-k\cdot (\overline{R})$  contains $\left( [-k,\tfrac{1}{2}-k] \cup [\tfrac{1}{2},1]\right)\times \{k\}$. Hence, for each open neighbourhood $U$ of $\overline{R}$ and each $k\in \N$, the points $x_k=(\tfrac{1}{2},k)$ are in the intersection of $U$ and $-k\cdot U$.
\end{example}

\end{remark}

We are interested in a condition on the group action that relates the relative compactness of the fundamental region to the compactness of the orbit space. One implication is trivial: if there is a relatively compact fundamental region $R$, i.e.~with compact closure, then the orbit space $X/G$ is compact because the projection map restricted to $\overline{R}$ is continuous and surjective.  However the converse, that a fundamental region is compact if the orbit space is compact is  false in general, as the following example illustrates.

	\begin{example}\label{ex_fundamental-region-not-locally-finite}
		Consider $X=\R$ and $G=\Z$. The action of $\Z$ on $\R$ (as an additive subgroup) is cocompact and properly discontinuous. The quotient $\R/\Z$ is the circle $\S^1$.
		Now we consider  the set $R = \bigcup_{n\in\N}(n+\frac n{n+1},n+\frac{n+1}{n+2})$. We note that since $\bigcup_{n\in\N} \overline{(\frac n{n+1},\frac{n+1}{n+2})} = [0,1)$, $R$ is a fundamental region for this action, yet $\overline R$ is not compact. 
		
This example can be modified to obtain a connected fundamental region, i.e.~a fundamental domain, which is not compact despite the action being cocompact. We do this by taking $X=\R^2$ and $G=\Z^2$.
		The quotient $X/G$ is the torus.
		Then we define $R:= \bigcup_{x\in(0,1)}\left(\{x\}\times (\frac1x,\frac1x+1)\right)$. Then $R$ is a connected fundamental region for a properly discontinuous and cocompact action, but $\overline R$ is not compact.
	\end{example}

Both fundamental regions in the example lack the property of 
 local finiteness, which is important when analysing the compactness of the orbit space in relation to the fundamental domain.

We conclude this section with few remarks on the existence of locally finite fundamental regions. We do not explore the details of these results in this paper, since we have no new results to contribute to the question of existence.

For many important continuous group actions, the existence of a locally finite fundamental region can be established. The prototypical example is an isometric group action on a metric space: here a locally finite fundamental region exists if the metric space is 
geodesically connected, geodesically complete and finitely compact (i.e.~if its closed balls are compact).
Geodesic completeness is a strong condition on a metric space. It means that any geodesic arc defined on a closed interval extends to a unique geodesic line $\lambda:\mathbb{R}\to X$.  We will briefly summarise the proof following \cite[Section~6.6]{Ratcliffe2019HyperbolicManifolds}.

First it is shown that a geodesically connected and complete metric space $(X,d)$ admits a point $x\in X$ for which the stabiliser in $G$ is trivial  \cite[Theorems~6.6.10 and 6.6.12]{Ratcliffe2019HyperbolicManifolds}. For such a point $x$ one considers the {\em Dirichlet domain}
\[D_{x}:=\{y\in X\mid d(x,y)<d(x,g(y))\text{ for all }e\not=g\in G\},\]
and shows that under the assumptions on $(X,d)$ that $D_x$ is a fundamental region that is  locally finite \cite[Theorem~6.6.13]{Ratcliffe2019HyperbolicManifolds}.

Kapovich's paper \cite{kapovich2023note} also provides  the following generalisation of this result to a class of actions which allow the construction of an invariant metric\footnote{After the first version of our article appeared on the arXiv,  Pierre-Louis Blayac outlined to us a different method of proving a version of  Kapovich's result. There is also an earlier  discussion in \cite{mof16}. Note  that  in the newest version of \cite{kapovich2023note} the relevant Section~7 about fundamental regions is completely rewritten and that the notion of {\em open fundamental region} in Definition~41 includes local finiteness.}.

	\begin{theoremnop*}[{\cite[Theorem 62]{kapovich2023note}}]
	Let $X$ be second countable, Hausdorff, connected, locally connected,  and locally compact. Let $G$ be a discrete and countable group that  acts freely and properly on $X$.
		Then there exists a locally finite fundamental region for the action.
	\end{theoremnop*}

\section{Results on fundamental regions}\label{result-sec}
In Section 6.6 of Ratcliffe's book \cite{Ratcliffe2019HyperbolicManifolds}, he presents theorems about fundamental domains of isometric group actions. Several of these theorems still make sense in the weakened context of continuous group actions, and here we present some of their modified proofs.
The motivation of this work is in results that relate the relative compactness of a fundamental region to the compactness of the orbit space. We consider the most useful results for applications to be the generalisations of Ratcliffe's Theorems 6.6.7 and 6.6.9, corresponding to our Theorems~\ref{Rat667} and \ref{Rat669}.
Note that only some of the proofs are analogous to Ratcliffe's arguments. When it is necessary to highlight the difference to the result for isometric actions, we state the relevant theorem from \cite{Ratcliffe2019HyperbolicManifolds} prior to stating our version of the result.

Our first  result is a direct generalisation of Theorem 6.6.3 in Ratcliffe's book \cite{Ratcliffe2019HyperbolicManifolds}.
	\begin{theorem}\label{thm_fundamental-region-implies-discrete}\label{Rat663}
		If a topological group $G$ acting continuously on a topological space $X$ admits a fundamental region, then $G$ has the discrete topology.
	\end{theorem}
	\begin{proof}
		If $G$ acts continuously, then the map $\Theta:G\times X\to X\times X$, given by $\Theta(g,x) = (x,gx)$ is continuous.
		Let $R$ be a fundamental region for this action, in particular we need that $R$ is an open set, with $\{gR\}_{g\in G}$ mutually disjoint.
		Then $\Theta^{-1}(R\times R) = \{\id\}\times R$, where $\id$ is the identity element of $G$.
		Hence $\{\id\}$ is open in $G$, and so $G$ carries the discrete topology.
	\end{proof}
\cref{thm_fundamental-region-implies-discrete} holds some interest in its own right. It states that for actions admitting fundamental regions, the discrete topology is the only relevant topology. It also tells us that for any continuous group action, a subgroup admitting a fundamental region must be discrete with the subset topology. In particular, this doesn't necessarily require the computation of the compact-open topology of the ambient group (nor any other particular topology).

A natural extension of this result also holds, stating that for actions admitting locally finite fundamental regions, the action must be properly discontinuous. This statement does not have a direct analogue in Ratcliffe's book, although we could consider it to also be an analogue of his Theorem 6.6.3, when combined with his Theorem 5.3.5.
\begin{theorem}\label{thmproperly-discontinuous-is-harmless}
	Let $R$ be a locally finite fundamental region for a group $G$ acting on a topological space $X$. Then the action is properly discontinuous.
\end{theorem}
\begin{proof}
	Let $x,y\in X$. Since $R$ is locally finite, we may take open neighbourhoods $U,V$ of $x,y$ respectively, such that $\{g\in G\ |\ U \text{ meets } g\overline R\}$, and $\{h\in G\ |\ V \text{ meets } h\overline R\}$ are both finite.

	We claim that the open neighbourhoods $U$ and $V$ exhibit proper discontinuity: Assume that for some $g\in G$, $gU$ meets $V$. Then, since the $g\overline R$ cover $X$, we find that $gU$ must meet $h_i \overline R$, for some $h_i$ in the finite set ${\{h\in G\ |\ V \text{ meets } h\overline R\}}.$

	But then $U$ meets $g^{-1} h_i \overline R$, and hence $g^{-1} h_i = g_j$ for some $g_j$ in the finite set $\{g\in G\ |\ U \text{ meets }g\overline R\}$.
	Hence $g = h_i g_j^{-1}$, i.e. $g$ is in the finite set 
		\[\{h\in G\ |\ V \text{ meets } h\overline R \} \cdot \{g\in G\ |\ U \text{ meets } g\overline R \}^{-1} .\]
	This shows that $gU$ meets $V$ for only finitely many $g$, and hence the action is properly discontinuous.
\end{proof}

\begin{remark}\label{ratremark} For completeness, let us mention that there are more results in Ratcliffe that have an immediate generalisation with an identical proof, which we therefore however do not repeat here. Examples are the following \cite[Theorems 6.6.4 and 6.6.5]{Ratcliffe2019HyperbolicManifolds}:

{\em Let  $G$ be a group of homeomorphisms of a topological space $X$, $R$ be a fundamental region with boundary $\partial R$, and $g\in G$.
\begin{enumerate}
\item If $g\neq\id$, then $\overline R\cap g\overline R \subset \partial R$.
\item If $g$ fixes a point of $X$, then $g$ is conjugate in $G$ to an element $h$ such that $h$ fixes a point of $\partial R$.
\end{enumerate}}
\end{remark}

Now, one of the most significant theorems in this section of Ratcliffe's book is the following.
	\begin{theoremnop*}[{\cite[Theorem 6.6.7]{Ratcliffe2019HyperbolicManifolds}}]
Let $G$ be a discontinuous group  of isometries of a metric space $X$ and  $R$  a fundamental region. Then the inclusion $\iota:\overline R\to X$ induces a continuous bijection $\kappa:\overline R/G\to X/G$, and $\kappa$ is a homeomorphism if and only if $R$ is locally finite.
	\end{theoremnop*}
In our version of the theorem, we need some topological conditions for one direction, however, these conditions are still more general than being a metric space.
	\begin{theorem}\label{thm_kappa-homeomorphism}\label{Rat667}\label{thm_Ratcliffe-LF-quotient-is-whole-quotient}
		Let $G$ be a topological group acting continuously on a topological space $X$. Let $R$ a fundamental region.
		Then the inclusion $\iota:\overline R\to X$ induces a continuous bijection $\kappa:\overline R/G\to X/G$. If $R$ is locally finite, then $\kappa$ is a homeomorphism.

		Further, if the action is properly discontinuous, $X$ is first-countable, and singletons in $X$ are closed, then the converse holds: If $\kappa$ is a homeomorphism, then $R$ is locally finite.
	\end{theorem}
	\begin{proof}
		We define $\kappa:Gx\cap \overline R\mapsto Gx$. Then $\kappa$ is immediately injective. It is surjective because $R$ is a fundamental region, and so every orbit has a member in $\overline R$.

		Let $\eta:\overline R\to \overline R/G$ be the projection map. So $\kappa$ is continuous if and only if $\kappa\circ\eta$ is continuous.
		Then $\kappa\circ\eta$ is continuous by the following commutative diagram
		$$\begin{tikzcd}
			\overline R \arrow[d, "\eta"] \arrow[r, "\iota"] & X \arrow[d, "\pi"] \\
			\overline R/G \arrow[r, "\kappa"]                & X/G               
		\end{tikzcd}$$
		So $\kappa$ is a continuous bijection.

		Let $R$ be locally finite. We show $\kappa$ is an open map. Let $\tilde U\subset \overline R/G$ be open. By definition, this is true if and only if $\eta^{-1}(\tilde U)$ is open, i.e.~$\eta^{-1}(\tilde U) = \overline R\cap U$, for some open set $U\subset X$.
		Define
		$$W:= \bigcup_{g\in G} g(\overline R\cap U).$$
		Note that $\pi(W) = \pi(\overline R\cap U) = \pi\circ \iota(\overline R\cap U) = \kappa\circ\eta(\overline R\cap U) = \kappa(\tilde U)$.
		So, noting that $\pi$ is an open map since we quotient by a group action by homeomorphisms, it suffices to prove that $W$ is open.

		Let $w\in W$. Since $R$ is locally finite, there is an open neighbourhood $\tilde V$ of $w$ that meets $\{g\overline R\}$ for only finitely many $g$. In particular, we have
		$$\tilde V\subset g_1\overline R\cup\hdots \cup g_m\overline R.$$
		Without loss of generality, there is a $k\in \{1, \ldots , m\}$ such that  $w\in g_i\overline R$ for all $i=1, \ldots ,k$. 		Then we consider the smaller open set
		$$
		V:= \tilde V \cap g_1U\cap\hdots \cap g_k U.
		$$
		We must verify that $V$ still contains $w$.
		Note that when $w\in g_i\overline R$, i.e. $g_i^{-1}w\in \overline R$, then $\kappa\circ \eta(g_i^{-1}w) = \pi\circ \iota(g_i^{-1}w) = \pi(g_i^{-1}w) = \pi(w)$.
		Then, using the fact from above that $\pi(W) = \kappa(\tilde U)$, together with the injectivity of $\kappa$, we have that $\eta(g_i^{-1}w) \in \tilde U$, and so $g_i^{-1}w \in \eta^{-1}(\tilde U)$.
		So, $g_i^{-1}w\in U$, hence $w\in g_iU$ for each $i$. Thus we have concluded that $V$ contains $w$.

		Then finally, $V\subset W$, since
		$$(g_1\overline R\cup\hdots \cup g_m\overline R)\cap g_1U\cap\hdots \cap g_mU \subset (g_1\overline R\cap g_1U) \cup\hdots\cup (g_m\overline R \cap g_mU).$$
		Hence $V$ is an open neighbourhood of $w$ contained in $W$, so $W$ is open. This completes the first part of the proof.

		Assume that $X$ is first countable and the action is properly discontinuous. For contradiction, also assume $\kappa$ is a homeomorphism and that $R$ is not locally finite. Let $y\in X$ be a point at which $\{g\overline R\}_{g\in G}$ is not locally finite. Note first of all that $y$ cannot be within $\hat gR$ for any $\hat g$, since $\hat gR$ is then a neighbourhood of $y$ that meets $\{g\overline R\}_{g\in G}$ for only finitely many (one) $g$.
		Since $X$ is first countable, we take a sequence $x_i\in R$ and $g_i$ distinct in $G$ such that $g_ix_i$ converges to $y$.
		We claim that $K:= \{x_i\}_{i\in\N}$ is closed.

		Let $x\in X\setminus K$. 
		Since the action is properly discontinuous, we take open neighbourhoods $U, V$ of $x, y$ respectively, such that $U$ meets $gV$ for only finitely many $g$.

		Since $g_ix_i\to y$, We have that $g_ix_i$ must eventually be contained in $V$. Hence for all but finitely many $i$, $x_i$ is contained in $g_i^{-1}V$. And so for all but finitely many $i$, $x_i$ is not in $U$.
		Then, since singletons are closed, 
		removing all points $x_i$ from the set $U$ gives
		an open neighbourhood of $x$, which does not meet $K$. So $K$ is closed.

		Since $K\subset R$ and $\eta$ is injective on $R$, we have that $\eta^{-1}(\eta(K))=K$ is closed, so $\eta(K)$ is closed in the final topology.
		Then assuming that $\kappa$ is a homeomorphism, we get further that $\pi(K) = \kappa(\eta(K))$ is closed.

		As $\pi$ is continuous, we have that $\pi(g_ix_i)\to \pi(y)$ and hence $\pi(x_i)\to \pi(y)$.
		But since $y$ is not in $gR$ for any $g$ and $K\subset R$, it must be that $\pi(y)$ is not an element of $\pi(K)$. This contradicts closedness of $\pi(K)$.
	\end{proof}

	Next we encounter two closely related theorems in Ratcliffe's book that do not easily generalise.  

		\begin{theoremnop*}[{\cite[Theorem 6.6.8]{Ratcliffe2019HyperbolicManifolds}}]
		Let $G$ be a group of isometries of a metric space $X$, $R$ a  locally finite fundamental region and  
 $x$  a boundary point of $R$.
		Then $\partial R\cap G x$ is finite and there is an $r>0$ such that, if $N(R,r)$ is the $r$-neighborhood of $R$ in $X$, then
		\[N(R,r)\cap G x = \partial R \cap G x.\]
	\end{theoremnop*}
	Ratcliffe's proof of this Theorem does not translate without additional structure, because in it, the metric is leveraged to make a choice of reasonable neighbourhood (in particular, the way the distance can be reduced to uniformly approach $\overline R$ is crucial to excluding $gx$'s away from the boundary).
	This theorem is used crucially in Ratcliffe's proof of the following theorem, a statement for which we would very much like an analogue as it relates the relative compactness of the fundamental region to that of the orbit space.

\begin{theoremnop*}[{\cite[Theorem 6.6.9]{Ratcliffe2019HyperbolicManifolds}}]
		Let $G$ be a discontinuous group of isometries of a locally compact metric space $X$ such that $X/G$ is compact and let $R$ be a fundamental region. Then $R$ is locally finite if and only if $\overline R$ is compact.
	\end{theoremnop*}

It turns out that for a generalisation of \cite[Theorem 6.6.8]{Ratcliffe2019HyperbolicManifolds}
we require the fundamental regions to be strongly locally finite, as defined in the introduction and in Definition~\ref{slf}, see Theorem~\ref{Rat668} below. 
Nevertheless, 
for an analogue of Ratcliffe's Theorem~6.6.9 we will provide a new proof that does not use the statement of  Theorem~\ref{Rat668}, and hence  does not require the stronger version of locally finite. This will be formulated in  Theorem~\ref{Rat669} below, which also implies that last statement of Theorem~\ref{main_theorem}.

Before we prove these results, we will establish the main part of Theorem~\ref{main_theorem} by showing  that for  cocompact actions local finiteness implies strong local finiteness
and that both are  equivalent to an even stronger condition for cocompact actions.
\begin{theorem}\label{thmcocompact-LF-FSA-are-equivalent}
	Let $G$ be a group of homeomorphisms of a topological space $X$. Let $R$ be a locally finite fundamental region.
	If the quotient $X/G$ is compact, then $\overline R$ has an open neighbourhood $U$ such that $gU$ meets $U$ for only finitely many $g\in G$.

	Hence for actions whose quotient is compact, the following are equivalent:
	\begin{itemize}
		\item $R$ is locally finite,
		\item $R$ is strongly locally finite,
		\item $\overline R$ has an open neighbourhood $U$ that meets $gU$  for only finitely many $g\in G$.
	\end{itemize}
\end{theorem}
\begin{proof}
	Immediately the third property implies the second, which implies the first. So all that needs to be shown is that locally finite implies the third property.

	Assume $\{g\overline R\}_{g\in G}$ is locally finite. We choose for each $x\in \overline R$, an open neigbourhood $U_x$ that meets $g\overline R$ for only finitely many $g$.

	Then $\{\pi(U_x)\}_{x\in\overline R}$ is an open cover of $X/G$, and so we can take a finite subcover $\{\pi (U_i)\}_{i\in I}$.

	Lifting this finite subcover we cannot conclude that the $U_i$'s cover $\overline R$, but their $G$-translates do: We have that $\{gU_i\}_{g\in G, i\in I}$ covers $\overline R$.

	But then by the local finiteness assumption, for each $i$, there is only finitely many $g$ such that $gU_i$ meets $\overline R$.

	Hence by considering the union of these $g$'s, we may assume we have two finite collections: $\{U_i\}_{i\in I}$, $\{g_a\}_{a\in A}$, such that:
	\begin{itemize}
		\item The finite set $\{g_aU_i\}_{a\in A,i\in I}$ covers $\overline R$, and
		\item if $g U_i$ meets $\overline R$, then $g\in \{g_a\}_{a\in A}$. 
	\end{itemize} 
	We claim that $V:= \bigcup_{a\in A,i\in I} g_aU_i$ is an open neighbourhood of $\overline R$ exhibiting the conclusion of the theorem: that $gV$ meets $V$ for only finitely many $g\in G$.

	Assume that $gV$ meets $V$ for some $g$. Then in particular, we have some $a,b,i,j$ such that $gg_aU_i$ meets $g_bU_j$.

	But note that $\{g_c\overline R\}_{c\in A}$ must cover $U_j$, since $\overline R$ is a fundamental region. And hence there must be some $c\in A$ such that $gg_aU_i$ meets $g_bg_c \overline R$, i.e. $U_i$ meets $g_a^{-1} g^{-1} g_b g_c\overline R$.

	Finally then again by the local finiteness assumption, it must be that this element $g_a^{-1} g^{-1} g_b g_c$ is equal to $g_d$ for some $d\in A$.

	Therefore $g$ is contained in $\{g_bg_cg_d^{-1}g_a^{-1}\}_{a,b,c,d\in A}$, which is a finite set.
\end{proof}

Now we turn to  an analogue of Ratcliffe's Theorem~6.6.8. With the stronger  assumption the proof is relatively straightforward: strong local finiteness overcomes the difficulty that arises from the action not necessarily being isometric.
	\begin{theorem}\label{Rat668}
		Let $G$ be a group of homeomorphisms of a topological space $X$. Let $R$ a strongly locally finite fundamental region and $x$ a boundary point of $R$.
		Then $\partial R\cap G x$ is finite.

		Further, if $X$ is $T_1$, then there is an open neighbourhood $U$ of $\overline R$ such that
		$$U\cap G x = \partial R\cap G x.$$
	\end{theorem}
	\begin{proof}
		Since $R$ is strongly locally finite, it is also locally finite, so for some open neighbourhood $V$ of $x$, $V$ meets $g\overline R$ for only finitely many $g\in G$. In particular then $\{x\}$ must also meet $g\overline R$ for only finitely many $g$, and so we conclude $gx$ is in $\overline R$ for only finitely many $g$.

		Now assume that $X$ is $T_1$, i.e.~that singletons in $X$ are closed, and let $\hat U$ be a strongly locally finite open neighbourhood of $\overline R$.  Then we define
		\[U:= \hat U\setminus \{gx\in X\setminus \partial R\},\]
which  is open since each $\{gx\}$ is closed and since only finitely many $gx$ are contained in $\hat U$.
	\end{proof}
Now we will establish an analogue for Ratcliffe's Theorem~6.6.8 with a proof that does not use   Theorem~\ref{Rat668} and hence  does not require the stronger version of locally finite. It will also imply the last statement in Theorem~\ref{main_theorem}.

	\begin{theorem}\label{thm_FSA-iff-FR-compact}\label{Rat669}
		Let $G$ be a group of homeomorphisms of a locally compact topological space $X$.
		Let $X/G$ be compact, and let $R$ be a fundamental region.
		If $R$ is locally finite, then $\overline R$ is compact.
		
		\nopagebreak
				Further, let $X$ be first countable, Hausdorff, and let the action be properly discontinuous. If $\overline R$ is compact, then $R$ is locally finite.
	\end{theorem}
	\begin{proof}
		We prove the second statement first. Under the assumptions, $X/G$ must be Hausdorff \cite[Theorem~11 and Lemma~9]{kapovich2023note}. Then $\kappa:\overline R/G\to X/G$ is a continuous bijection from a compact space to a Hausdorff space and so is a homeomorphism. Then we apply \cref{thm_kappa-homeomorphism} to conclude that $R$ is locally finite.

		Now we prove the first statement.
		Let $X/G$ be compact. Since $X$ is locally compact, there is a compact set $K$ such that $\pi(K) = X/G$ (which can be found by taking a finite subcover of $X/G$ by $\pi(\operatorname{interior}(K_x))$ for compact neighbourhoods $K_x$ of $x\in X$). We now show that we can cover $\overline R$ by a finite union of $gK$.

		By assumption, $R$ is locally finite, so each $x\in K$ has an open neighbourhood $V_x$, such that $V_x$ meets $g\overline R$ for only finitely many $g$.
		Since these $V_x$'s cover the compact set $K$, we take a finite subcover. Then, by considering the union of this subcover, we have an open neighbourhood $V$ of $K$, such that $V$ meets $g\overline R$ for only finitely many $g$, say $g\in \{g_1,\ldots, g_m\}$.

We note that if $gK$ meets $\overline R$, then $V$ meets $g^{-1}\overline R$, so that, $g^{-1}\in \{g_1, \ldots, g_m\}$.
As a consequence, 
$gK$ only meets $\overline R$ for finitely many $g$. But by assumption, $\pi(K) = X/G$, and so in particular $GK$ must cover $\overline R$. Hence $\overline R$ is contained in the finite union $\bigcup_i g_i^{-1}K$, and so $\overline R$ is compact.
	\end{proof}

	Together with Theorem~\ref{thmcocompact-LF-FSA-are-equivalent}, we conclude Theorem~\ref{main_theorem} as a corollary.
\section{Applications}\label{appl}
	In this section, we provide two applications of these results and the notion of strong local finiteness to the geometric problems that motivated this investigation. In particular, we will prove Proposition~\ref{prop_homoth-essential-implies-no-FSA-FR} and hence Corollary~\ref{introcor}.

	First, we use Theorem~\ref{thm_FSA-iff-FR-compact} to prove a general result about a group that acts by translations in one component. It gives an efficient criterion for the cocompactness of such actions, which we used in \cite{LeistnerTeisseire22} when  studying cocompact actions on indecomposable Lorentzian symmetric spaces of non-constant sectional curvature. Note that the following result can be straightforwardly generalised to an action of $\Z^n$ on a product of $X$ with $\R^n$.

\begin{proposition}\label{prop_quotienting-with-a-translation-component}
		For a locally compact topological space $X$  consider $X\times \R$ with the product topology and let $\mathrm{pr}_\R:X\times \R\to \R$ be the projection onto $\R$.
If 		$\phi$ is a homeomorphism of $X\times \R$ such that 
		\[\mathrm{pr}_\R\circ \phi\,(x,t)= t+c,\quad\text{ for some $c>0$},\]
		then $X\times (0,c)$ is a strongly locally finite fundamental region for the action of the group $\<\phi\>$. In particular, then the quotient $X\times \R/\<\phi\>$ is compact if and only if $X$ is compact.
	\end{proposition}
	\begin{proof}
		We define
		\begin{align*}
			R&:= X\times(0,c) = \{(x,t)\in X\times \R\ |\ 0<t<c\}
		\end{align*}
		and show $R$ is a strongly locally finite fundamental region.
		Note that \[\phi^m R = X\times(mc,mc+c),\] and hence $\phi^mR$ does not meet $R$ for $m\neq 0$.
		Similarly, $\phi^m\overline R = X\times[mc,mc+c]$ and so $\bigcup_{m\in\Z} \phi^m\overline R = X\times \R$ so that $R$ is a fundamental region.
		Define an open neighbourhood of $\overline R$,
		$$U:= X\times(-c,2c).$$
		Note that $\phi^m U = X\times(mc-c, mc+2c)$ and so if $\phi^mU$ meets $U$, then $m\in\{-2,-1,0,1,2\}$.
		Therefore since $GU$ covers the space, $R$ is strongly locally finite.

		Note that $X\times \R$ is locally compact because $X$ and $\R$ are locally compact. Then by Theorem~\ref{thm_FSA-iff-FR-compact}, $(X\times \R)/\<\phi\>$ is compact if and only if $X\times [0,c]$ is compact, which is true if and only if $X$ is compact.
	\end{proof}
	Note further that, applying \cref{thm_Ratcliffe-LF-quotient-is-whole-quotient}, $(X\times\R)/\<\phi\>$ is homeomorphic to $\overline R/\<\phi\>$, which is simply $X\times [0,c]/\sim$, where $\sim$ is a gluing rule on the boundary $X\times\{0\}\sim^\phi X\times\{c\}$.

\bigskip

Finally,
we will prove Proposition~\ref{prop_homoth-essential-implies-no-FSA-FR} and hence Corollary~\ref{introcor}, both relevant to the results in  \cite{LeistnerTeisseire22}. For a semi-Riemannian manifold, recall the definition of inessential and essential homotheties of $(M,g)$ from the introduction. It is easy to see and shown in  \cite[Proposition 2.5]{LeistnerTeisseire22} that a non-trivial homothety with a fixed point must be essential. Moreover, for indecomposable Lorentzian symmetric spaces with nonconstant sectional curvature, so-called {\em Cahen--Wallach spaces}, the converse holds \cite[Theorem 1.2]{LeistnerTeisseire22}, however for general semi-Riemannian manifolds this is not clear. Proposition~\ref{prop_homoth-essential-implies-no-FSA-FR} in the introduction is a step towards this converse by showing that the group generated by an essential homothety cannot admit a strongly locally finite fundamental region. This situation is summarized in the following diagram for any strict homothety $\phi$:
$$
\begin{tikzcd}[column sep=-10em]
	& \<\phi\>\text{ has a strongly locally finite fundamental region} \arrow[rd, "\text{Proposition \ref{prop_homoth-essential-implies-no-FSA-FR}}", Rightarrow] & \\
	\phi\text{ has no fixed points} & & \phi\text{ is not essential} \arrow[ll, Rightarrow]
\end{tikzcd}
$$
The missing implication on the left, which would complete the equivalence, is demonstrated in \cite[Theorem 4.4.1]{stuart-thesis} for homotheties of Cahen--Wallach spaces, and conjectured in general.
%
	\begin{proof}[Proof of Proposition~\ref{prop_homoth-essential-implies-no-FSA-FR}]
We will construct a function $f$ such that $f\circ \phi + s = f$.
By assumption, $\overline R$ has an open neighbourhood $U$ such that $\{\phi^iU\}_{i\in \Z}$ is locally finite. We show that we can construct a partition of unity subordinate to this cover, with the additional property that $f_i = f_{i+1}\circ\phi$.

		Take a bump function $\hat f_0\geq0$, whose support is in $U$, such that $\hat f_0|_{\overline R} \equiv 1$. Then for each $i\in\Z$, define
		$$\hat f_i:= \hat f_0\circ \phi^{-i}.$$
		Note that the support of $\hat f_i$ is contained within $\phi^i(U)$ and $\hat f_i|_{\phi^i(\overline R)}\equiv 1$. Note also that
\[
		\hat f_i = \hat f_0\circ \phi^{-i-1}\circ\phi
= \hat f_{i+1}\circ \phi.
\]
		Then we define
		$$f_i:= \hat f_i/\Big(\sum_{j\in\Z} \hat f_j\Big)$$
		Since the support of $\hat f_i$ are locally finite, the sum is finite in an open neighbourhood of each point, so $f_i$ is smooth since $\hat f_j$ are all smooth. Since $\phi^j(\overline R)$ cover $M$ and $\hat f_j\geq 0$, we see that $\sum_{j\in\Z}\hat f_j$ has no zeros, so $f_i$ is defined everywhere.
		Hence $\{f_i\}$ is a partition of unity, which satisfies $f_i = f_{i+1}\circ \phi$.

		Now we define the rescaling $f$, so that $\phi$ is an isometry of $\mathrm{e}^{2f}g$. Let $\phi^*g=\mathrm{e}^{2s}g$ and define
		$$f:= -s\sum_{i\in\Z}(if_i).$$
		The function $f$ is smooth since since each $f_i$ is smooth and their supports are locally finite. The function $f$ also has the property that
\[
			f\circ\phi = -s\sum_{i\in\Z}(if_i\circ\phi)
				= -s\sum_{i\in\Z}(if_i) -s \sum_{i\in \Z} f_i
				= f -s.
\]
		Hence we see that
\[
			\phi^*(\mathrm{e}^{2f}g) = \mathrm{e}^{2f\circ\phi}\phi^*g
			= \mathrm{e}^{2f\circ\phi}\mathrm{e}^{2s}g\\
				= \mathrm{e}^{2f}g,
\]
		and so $\phi$ is not essential.
	\end{proof}
Together with Theorem~\ref{main_theorem}	we obtain Corollary~\ref{introcor}
 for cyclic groups of homotheties acting with  a compact manifold as  orbit space.


\bibliographystyle{abbrv}

\providecommand{\MR}[1]{}\def\cprime{$'$} \def\cprime{$'$} \def\cprime{$'$}

\end{document}